%% file: Fuhrmann-Wang-Rectifiability-of-a-class-of-invariant-measures-Final-15-06-2016.tex
\documentclass[12pt,oneside,a4paper]{amsart}

\usepackage[left=3cm,right=3cm,top=3cm,bottom=4cm,includeheadfoot]{geometry}
\usepackage{german, amsmath, amssymb, amsthm, latexsym, fancyhdr, enumerate, txfonts, expdlist, hyperref, mdwlist, mathrsfs, tikz, todonotes,pifont}
\usepackage[active]{srcltx}
\usepackage[T1]{fontenc}
\usepackage[latin1]{inputenc}
\usepackage{subfig}
\input{bibstandard}

\theoremstyle{plain}
\newtheorem{thm}{Theorem}[section]
\newtheorem{lem}[thm]{Lemma}
\newtheorem{prop}[thm]{Proposition}
\newtheorem{cor}[thm]{Corollary}

\theoremstyle{definition}
\newtheorem{defn}[thm]{Definition}

\theoremstyle{remark}
\newtheorem*{rem}{Remark}

\numberwithin{equation}{section}
\numberwithin{thm}{section}

\makeatletter
\def\@settitle{\begin{center}%
  \baselineskip14\p@\relax
    \normalfont\bf\LARGE%
  \@title
  \end{center}%
}
\makeatother

\title{\Large\textsc{Rectifiability of a class of invariant measures with one non-vanishing Lyapunov exponent}}
\author{G.~Fuhrmann}
\address{Institute of Mathematics, Friedrich-Schiller-Universit\"at Jena, Germany.}
\email{gabriel.fuhrmann@uni-jena.de}
\author{J.~Wang}
\address{Department of Applied Mathematics, Nanjing University of Science and
Technology, Nanjing 210094, China}
\email{jingwang018@gmail.com}
\thanks{Both authors would like to thank Tobias J\"ager for introducing them to the problem considered in this article. They would further like to thank
Tobias J\"ager and Maik Gr\"oger for related discussions.}
\thanks{J. W. was supported by a research fellowship of the Alexander von Humboldt-Foundation.}
\date{\today}

\begin{document}
\begin{abstract}
  We study order-preserving $\mc C^1$-circle diffeomorphisms driven by irrational \mbox{rotations} with a Diophantine rotation number.
  We show that there is a non-empty open set of one-parameter families of such diffeomorphisms
  where the ergodic measures of nearly all family members are one-rectifiable, that is, 
  absolutely continuous with respect to the restriction of the one-dimensional Hausdorff measure to a countable union of Lipschitz graphs.
\end{abstract}
\maketitle
\section{Introduction}
A fascinating aspect of the theory of dynamical systems is its contribution to the understanding of how complex behaviour and complex structures
originate from simple rules.
One phenomenon which fits perfectly in this category are so-called \emph{strange non-chaotic attractors}:
sets of a ``strange'' geometry which are invariant
and attracting under the dynamics of certain zero-entropy extensions of irrational rotations.
Under fairly general conditions, the strangeness of these invariant sets can be quantified
in terms of a dimension gap of the associated physical measures:
While they are of full support, they are--under mild assumptions--exact dimensional with a pointwise dimension equal to $1$.
The latter follows from a strong result by Ledrappier and Young \cite{LedrappierYoung1985} and is in perfect agreement with a famous conjecture by Yorke et al. \cite{fredericksonkaplanyorke1983}
(see also \cite{Ledrappier1981,Young1982Dimension}).

With this article, we show that in many cases the measures corresponding to strange non-chaotic attractors are in fact one-rectifiable, that is, 
they are absolutely continuous with respect to the restriction 
of the one-dimensional Hausdorff measure to a countable union of Lipschitz graphs (see Section~\ref{sec: rectifiable measures} for 
the exact definition).
Similar results have already been obtained in previous studies \cite{GroegerJaeger2012SNADimensions,fuhrmanngrogerjager14}.
However, the underlying geometric picture of the proof in the present case differs to a large extent and makes the authors believe that
rectifiability should be expected also in more general situations.

Throughout this work, we consider diffeomorphisms homotopic to the identity on $\T^2$ given by skew-products of the form
\begin{align}\tag{$\ast$}\label{eq: defn skew product introduction}
 f\:\T^2\to\T^2, \qquad\quad (\theta,x)\mapsto (\theta+\w, f_{\theta}(x)),
\end{align}
where the forcing frequency $\w\in\T^1\=\R/\Z$ is assumed to be irrational and $\T^2\ni(\theta,x)\mapsto f_{\theta}(x)\in\T^1$ is $\mc C^1$.
Occasionally, we may refer to maps of the form \eqref{eq: defn skew product introduction}
as \emph{quasiperiodically forced (qpf) circle maps}.
We are interested in the \emph{invariant graphs} of such qpf circle maps.
These are defined to be measurable functions $\phi\:\T^1\to\T^1$ such that
\[
 f_\theta(\phi(\theta))=\phi(\theta+\w)
\]
for $\textrm{Leb}_{\T^1}$-almost every $\theta\in \T^1$, where $\textrm{Leb}_{\T^1}$ denotes the Lebesgue measure on $\T^1$.
Note that in this case, the graph $\Phi\=\{(\theta,\phi(\theta))\:\theta\in \T^1\}$--which we always denote by the corresponding capital letter--is
in fact almost surely invariant under $f$, meaning
that there is a full-measure set $\Omega\ssq\T^1$ such that $f(\Phi\cap(\Omega\times \T^1))=\Phi\cap(\Omega\times \T^1)$.
We should remark that--by a slight abuse of terminology--we refer by graph to both the map $\phi$ and the point set $\Phi$.
Observe further that we identify invariant graphs if they coincide $\textrm{Leb}_{\T^1}$-almost surely.

It is natural to ask whether a given invariant graph $\phi$ attracts or repels nearby orbits.
The answer to this question is provided by its \emph{Lyapunov exponent}
\[
 \lam(\phi)\=\int_{\T^1}\! \log|\d_x f_\theta(\phi(\theta))| \,d\theta.
\]
If $\lam(\phi)<0$, the graph is attracting; if $\lam(\phi)>0$, the graph is repelling;
for the details, we refer the readers to \cite[Proposition~3.3]{JagerNonLin}.\footnote{Note that in the case when $f$ is $\mc C^{1+\alpha}$, this follows from Pesin theory
(cf. the supplement in \cite{KatokHasselblatt}).}

The dynamical importance of invariant graphs becomes apparent through their close relation
to the invariant measures of the systems under consideration:
To each invariant graph $\phi$, we can associate a measure $\mu_\phi$
given by
\[
 \mu_\phi(A)=\textrm{Leb}_{\T^1}(\pi_1(A\cap \Phi))
\]
for every Lebesgue-measurable set $A\ssq\T^2$, where $\pi_1$ is the projection to the first coordinate.
It is easy to see that $\mu_{\phi}$ is $f$-\emph{invariant}, that is, $\mu_{\phi}(A)=\mu_\phi(f^{-1}(A))$ for all Lebesgue-measurable sets $A$ and \emph{ergodic}, that is,
$f^{-1}(A)=A$ only if $\mu_{\phi}(A)$ equals $0$ or $1$.
In fact, if $f$ is not \emph{uniquely ergodic} (that is, if there are at least two distinct ergodic measures), the converse of this observation is also true if 
we allow for multi-valued invariant graphs (see \cite[Theorem~4.1]{Furstenberg1961}).

Our goal is to study the geometry of those ergodic measures which are supported on a particular kind of invariant graphs.
\begin{defn}
 We say an invariant graph $\phi\:\T^1\to\T^1$ is a \emph{strange non-chaotic attractor (SNA)} and \emph{repeller (SNR)} if it
 is attracting and repelling, respectively, and if it is
 non-continuous, that is, there is no continuous representative in its equivalence class.
\end{defn}
The above notion goes back to an article by Grebogi et al. from 1984, where numerical evidence and heuristic arguments for the existence
of SNA's are found for a rather particular class of skew-product systems on $\T^1\times\R$ (cf. \cite{grebogi/ott/pelikan/yorke:1984,Keller}).
However, rigorous results establishing the existence of SNA's (at least implicitly) had already been derived before \cite{millionscikov,Vinograd,herman1983}
in the context of certain quasiperiodic $\textrm{SL}(2,\R)$-cocycles, where the presence of SNA's is equivalent to the non-uniform hyperbolicity of the respective cocycle
(for a detailed discussion of this relation, see \cite[Section~1.3.2]{JagerAMS}).
In this setting, Young \cite{young:1997} and Bjerkl\"ov \cite{bjerkloev:2005,bjerkloev:2007} developed powerful methods--in the spirit of the multiscale
analysis and parameter exclusion techniques by Benedicks and Carleson \cite{benedicks/carleson:1991}--to examine the occurrence and properties of SNA's.
These methods had later been adapted to non-linear systems (such as \eqref{eq: defn skew product introduction}) in
\cite{Jager,JagerETDS,fuhrmann2014,fuhrmanngrogerjager14}.

A natural context in which SNA's arise can be found in the study of \emph{mode-locking} phenomena for qpf circle maps \cite{JagerETDS,jagerwang2015}.
Mode-locking (sometimes also referred to as \emph{frequency locking}) is best known as a phenomenon occurring in families $(g_\tau)_{\tau\in[0,1]}$ of continuous orientation-preserving circle maps
and describes the situation in which the rotation number $\rho(\tau)$ (i.e., the average speed by which points move
on $\T^1$ under the dynamics of $g_\tau$ [see, e.g., \cite[Proposition~11.1.1.]{KatokHasselblatt}]) is a
devil's staircase, that is, it is locally constant on an open
and dense subset while it increases continuously from $0$ to $1$ over the unit
interval (cf., e.g., \cite[Proposition~11.1.11.]{KatokHasselblatt}).
The paradigm example for the abundance of mode-locking
is certainly provided by the Arnold circle map
\[
  f_{\alpha,\tau} : \kreis\to\kreis, \qquad
   x\mapsto x+\tau+\frac{\alpha}{2\pi}\sin(2\pi x) \ \bmod 1,
\]
where $[0,1]\ni\tau\mapsto \rho(\tau)$ is a devil's staircase for all $\alpha\in(0,1]$.

The Arnold circle map gives an understanding of frequency-locking phenomena
occurring in a variety of real-world situations ranging from damped
pendula and electronic oscillators
\cite{Ding1986NonlinearOscillators} as well as the heart-beat \cite{arnold:1991}
through to paradoxical neural behaviour
\cite{perkeletal:1964,CoombesBressloff1999ModeLocking}.
Against the background of these applications, it is desirable to study mode-locking in dynamically more complicated situations
than the present one, where instead of the rotation numbers of families of circle maps, the \emph{fibre-wise rotation numbers} of
families of \emph{forced} circle maps of the form \eqref{eq: defn skew product introduction} are considered (see \cite{jagerwang2015,jagerwang2016}).
However, such families naturally yield SNA/SNR-pairs if we assume the forcing frequency $\w$ to be \emph{Diophantine}, that is,
poorly approximable by rational numbers (see Section~\ref{subsec: basic setting and notation}).
\begin{thm}[{\cite[Theorem~3.1]{jager2013}}]\label{t.modelocking}
  Given Diophantine $\omega\in\T^1$ and $\delta>0$, there exists a non-empty $\mc C^1$-open subset
  \[
  \mc U\ssq \left\{(f_\tau)_{\tau\in\T^1}\: \ f_\tau \textrm{ is of the form \eqref{eq: defn skew product introduction}  } \textrm{ and }
  (\tau,\theta,x)\mapsto f_\tau(\theta,x)  \textrm{ is } \mathcal C^1\ \textrm{for all}\ \tau\in\T^1\right\}
  \]
  with the following property.
  For all
  $(f_\tau)_{\tau\in\kreis}\in\mc U$ there is a set $\Lambda\ssq\kreis$
  with $\textrm{Leb}_{\T^1}(\Lambda)\geq1-\delta$ such that
 for all $\tau\in\Lambda$, the map $f_\tau$ has a
  (unique) SNA $\phi_\tau^+$ and SNR $\phi_\tau^-$ and the dynamics of $f_\tau$ are minimal.
\end{thm}
It is not only, but in particular, the situation of the last statement in which we describe the geometry of the ergodic measures
associated to the SNA $\phi^+$ and SNR $\phi^-$, respectively.
This description yields that the measures $\mu_{\phi^+}$ and $\mu_{\phi^-}$ are $1$-rectifiable, that is,
they are absolutely continuous with respect to the restriction of the $1$-dimensional Hausdorff-measure (on $\T^2$)
to a countable union of Lipschitz graphs
(see Section~\ref{sec: rectifiable measures} for the details).
In this terms, our main result reads as follows (see Theorem~\ref{prop: dimensions_subgraphs} for the full statement).
\begin{thm}\label{thm: main introduction}
  Given Diophantine $\omega\in\T^1$, $\delta>0$ and a family of qpf circle maps $(f_\tau)_{\tau\in\T^1}\in\mc U$, consider $f_\tau$
  for a parameter $\tau\in \Lambda$, where $\mc U$ and $\Lambda$ are as in Theorem~\ref{t.modelocking}.
  Then $\mu_{\phi_\tau^+}$ and $\mu_{\phi_\tau^-}$ are one-rectifiable.
\end{thm}
Observe that we hence obtain the afore-mentioned dimension gap as an immediate corollary:
While the rectifiability implies that the pointwise dimension of $\mu_{\phi_\tau^\pm}$ (for $\tau\in \Lambda$) equals $1$ almost surely
(see Corollary~\ref{cor: dimensions rectifiable measure}),
the box dimension of $\phi_\tau^\pm$ (and hence of the support of $\mu_{\phi_\tau^\pm}$) equals $2$--the dimension of the phase space $\T^2$.
The latter is a result of the stability of the box dimension
under taking closures and the denseness of $\phi^\pm_\tau$ in $\T^2$ which follows immediately from the minimality of $f_\tau$.

We want to remark that under the additional assumption of $f$ being $\mc C^2$, this dimension gap already follows from \cite[Corollary~I]{LedrappierYoung1985} where an upper bound
for the pointwise dimension is proven to be given by the Lyapunov dimension which is $1$ in the present case.
Similar arguments, based on the findings in \cite{Ledrappier1981}, had been applied in \cite{KimKimHuntOtt2003} to obtain the information dimension of
robust strange non-chaotic attractors.

However, the main point of the present work is to show the high degree of regularity of the measures 
$\mu_{\phi_\tau^\pm}$ mentioned above.
To this end, we have to decompose the graphs $\phi^+$ and $\phi^-$ (almost everywhere) in countably many
Lipschitz continuous graphs (cf. Proposition~\ref{prop: decomposition in lipschitz graphs}).
On a combinatorial level, the strategy we pursue has been applied successfully to ergodic measures supported on SNA/SNR-pairs
that occur in so-called \emph{saddle-node bifurcations} of \emph{qpf monotone interval maps} \cite{fuhrmanngrogerjager14}.
These are skew-products similar to \eqref{eq: defn skew product introduction}
but defined on $\T^1\times[0,1]$ and such that the maps $f_\theta(\cdot)$ are monotonously increasing (for each fixed $\theta\in\T^1$).
We will thus be able to recycle the rather technical combinatorial findings of \cite{fuhrmanngrogerjager14}.

On a geometric level, however, both situations are completely different: in \cite{fuhrmanngrogerjager14},
the monotonicity allowed for a point-wise approximation of the SNA (and SNR, respectively) by $\mc C^1$-curves.
The convergence of the $\mc C^1$-curves even turned out to be \emph{uniform} on sets of measure arbitrarily close to $1$ which hence yielded the desired
decomposition.
In the present situation, such an approximation seems out of reach.
As a result, we have to implement a local approach.
For a sketch of this local strategy, see Section~\ref{sec: rectifiability}. 
The details are given in the last section.

The fact that we observe rectifiability even beyond the possibility of obtaining the invariant graphs as limits of $\mc C^1$-curves
makes the authors believe that this property is verified by a larger class of 
invariant ergodic measures
with full support, zero entropy, and only one non-vanishing Lyapunov exponent.

Let us conclude this paragraph with some explicit examples of skew-product families our results apply to.
We want to remark, that these examples are discussed in further detail in \cite{Jager,JagerETDS}.
For $x\in \T^1$, let $\hat x\in(-1/2,1/2]$ be a lift of $x$, that is,
$\pi(\hat x)=x$, where $\pi\: \R\to \T^1$ denotes the canonical projection.
For $q\geq 2$ and $\alpha>0$, set $h_q\=\pi(a_q(\alpha \hat x)/2a_q(\alpha/2))$ with
\[
 a_q(x)\=\int_{0}^x\! 1/(1+|\zeta|^q)\,d\zeta.
\]
It is straightforward to see that
\begin{align}\tag{$\ast\ast$}\label{example 1}
 g_{q,\tau}\:\T^2\ni (\theta,x)\mapsto (\theta+\w, h_q(x)+V(\theta)+\tau)
\end{align}
is of the form \eqref{eq: defn skew product introduction} for each $\tau\in[0,1]$ and $V\:\T^1\to\T^1$.
In fact, for each $q$ there are appropriate $V$ such that \eqref{example 1} lies in the set $\mc U$ of Theorem~\ref{t.modelocking} if $\w$ is Diophantine and $\alpha$ is large
enough--depending on $\delta,\ \w,\ q$ and $V$--and we thus obtain the $1$-rectifiability
of the ergodic measures for all $\tau$ in a set of Lebesgue measure at least $1-\delta$.
Notice that $a_2(x)=\arctan(x)$ so that \eqref{example 1} in particular contains the projective action of the
$\textrm{SL}(2,\R)$-cocycle over the irrational rotation by $\w$ associated to
\[
 A(\theta)=
 \textrm{R}_{V(\theta)+\tau}\cdot
 \begin{pmatrix}
  \alpha & 0\\
  0 & 1/\alpha
 \end{pmatrix},
\]
where $\textrm{R}_\varphi$ denotes the rotation matrix by angle $\varphi$.
Here, a possible choice is $V(\theta)=\cos(2\pi\theta)$, for example.

We would further like to mention that, in principal, our
arguments also show the $1$-rectifiability of the invariant measures of the driven Arnold circle map
\[
  f_{\alpha,\beta,\tau} : \T^2\to\T^2,
  \quad (\theta,x)\mapsto \left(\theta+\w, x+\tau+\frac{\alpha}{2\pi}\sin(2\pi x) + V_\beta(\theta)\ \bmod 1\right)
\]
  for appropriate $V_\beta$ and $|\alpha|\leq 1$.
  Strictly speaking, some modifications are needed to
  include this case: the derivative of the fibre map $f_{\alpha,\beta,\tau,\theta}$ with respect to $x$
  remains bounded by 2 for any fixed $\theta\in\T^1$ in the invertible regime $|\alpha|\leq 1$.
  However, our proofs hinge on high expansion rates in the $x$-direction.
  To bypass this problem, we would have to require a
  special shape of the forcing function [a suitable choice is $V_\beta(\theta)=\arctan(\beta\sin(2\pi\theta))/\pi$] and a largeness
  assumption on the additional parameter $\beta$.
  We omit the technicalities of the discussion of this special case and refer the interested readers to \cite{Jager,JagerETDS} for the details.

  The required prerequisites of our investigation are presented in the next section.
  Section 3 yields the proof of our main result under the assumption of a technical proposition whose
  proof is postponed to the last section.

\section{Preliminaries}
In the first subsection, we shortly collect basic facts from geometric measure theory.
In the second subsection, we provide a precise description of
those systems we consider throughout this work, and formulate our main result.

\subsection{Rectifiable measures}\label{sec: rectifiable measures}
We provide the definition and a few properties of rectifiable measures where we mainly follow
\cite{AmbrosioKirchheim2000}.

Let $Y$ be a metric space.
We denote the diameter of a subset $A\subseteq Y$ by $|A|$.
For $\varepsilon>0$,
we call a finite or countable collection $\{A_i\}$ of subsets of $Y$
an {\em $\varepsilon$-cover} of $A$ if $|A_i|\leq\varepsilon$ for each
$i$ and $A\subseteq\bigcup_i A_i$.
\begin{defn}
For $A\subseteq Y$, $s\geq 0$ and $\varepsilon>0$, we define
\[
\mathcal H_\varepsilon^s(A)\=\inf\left\{\left.\sum\limits_i
    \left|A_i\right|^s \ \right|\ \{A_i\}\text{ is an
    $\varepsilon$-cover of $A$}\right\}
\]
and call
\[
	\mathcal H^s(A)\=\lim\limits_{\varepsilon\to 0} \mathcal H_\varepsilon^s(A)
\]
the \emph{$s$-dimensional Hausdorff measure} of $A$.
\end{defn}
\begin{defn}
  For $d\in\N$, we call a Borel set $A\subseteq Y$ \emph{countably
    $d$-rectifiable} if there exists a sequence of Lipschitz
  continuous functions $(g_i)_{i\in\N}$ with $g_i:A_i\subseteq\R^d\to
  Y$ such that $\mathcal H^d(A\backslash\bigcup_i g_i(A_i))=0$. A
  finite Borel measure $\mu$ is called \emph{$d$-rectifiable} if
  $\mu=\Theta\left.\mathcal H^d\right|_A$ for some countably
  $d$-rectifiable set $A$ and some Borel measurable density
  $\Theta:A\to[0,\infty)$.
\end{defn}
Observe that, by the Radon-Nikodym theorem, $\mu$ is $d$-rectifiable
if and only if $\mu$ is absolutely continuous with respect to
$\left.\mathcal H^d\right|_A$ where $A$ is a countably $d$-rectifiable
set.
\begin{thm}[{\cite[Theorem 5.4]{AmbrosioKirchheim2000}}]
  For a $d$-rectifiable measure $\mu=\Theta\left.\mathcal
    H^d\right|_A$, we have
\[
\Theta(x)=\lim\limits_{\varepsilon\to
  0}\frac{\mu(B_\varepsilon(x))}{V_d\varepsilon^d},
\]
for $\mathcal H^d$-a.e.\ $x\in A$, where $V_d$ is the volume of the
$d$-dimensional unit ball. The right-hand side of this equation is
called the $d$-density of $\mu$.
\end{thm}
From the last theorem, we can deduce that the $d$-density exists and
is positive $\mu$-almost everywhere for a $d$-rectifiable measure
$\mu$. This directly implies the next corollary.

For $x\in Y$ and $\varepsilon>0$, let
$B_\varepsilon(x)$ be the open ball around $x$ with radius
$\varepsilon>0$.
\begin{cor}\label{cor: dimensions rectifiable measure}
  A $d$-rectifiable measure $\mu$ is \emph{exact dimensional} with $d_\mu=d$, that is,
  the \emph{pointwise dimension}
\begin{align*}
  d_\mu(x)\=\lim\limits_{\varepsilon\to 0}\frac{\log\mu(B_\varepsilon(x))}{\log\varepsilon}\\
\end{align*}  
  exists and equals $d$ $\mu$-almost surely.
\end{cor}
\begin{rem}
Note that if $\mu$ is exact dimensional, then in the setting of
separable metric spaces several other dimensions of $\mu$ coincide
with the pointwise dimension \cite{Zindulka2002,Young1982Dimension,Pesin1993}.
\end{rem}

\subsection{Statement of the main result}\label{subsec: basic setting and notation}
The aim of this section is to formulate a number of assumptions that define a set $\mathcal{V}_\w$ of skew-products
which guarantee the existence of SNA/SNR-pairs whose associated invariant measures are $1$-rectifiable.
In particular, the set $\mc V_\w$ will comprise those members of the families considered in Theorem~\ref{t.modelocking} for which Theorem~\ref{t.modelocking}
ensures the existence of an SNA.

Principally speaking, it would be possible to define ${\mc V}_\w$
by means of explicit $\mathcal{C}^1$-estimates only (cf. Proposition~\ref{prop: tobias cmp} and Proposition~\ref{prop: tobias etds} below and
the corresponding references).
However,
besides some of these estimates,
our investigation builds on particular dynamical properties--foremost some slow
recurrence conditions for certain critical sets defined in the
multiscale analysis carried out in \cite{Jager,jager2013}--which are already a result of the collection of these explicit estimates.
In order to avoid the redundance of proving these properties once more and for the reader's convenience, we
will define ${\mc V}_\w$ in a partially intrinsic and
somewhat abstract way by means of those $\mathcal{C}^1$-estimates that are needed for our purposes and
by means of the required dynamical behaviour.
However, the important fact is that for $\w$ being Diophantine, the set ${\mc V}_\w$ is rich (cf. Proposition~\ref{prop: tobias cmp} and Proposition~\ref{prop: tobias etds})
and contains the examples of the form \eqref{example 1} discussed in the introduction.

Let $\mathcal{F}:=\{f\in \textrm{Diff}^1(\T^2)\ |\ \pi_1\circ f=\pi_1\}$, where $\textrm{Diff}^1(\T^2)$
denotes the group of diffeomorphisms of the two-torus $\T^2$ which are homotopic to the identity, and $\pi_i$ is the projection to the respective
coordinate. Note that for $F\in\mathcal{F}$ we have $F(\theta,x)=(\theta,f_\theta(x))$ where $f_\theta(\cdot)=\pi_2\circ F(\theta,\cdot)$,
such that we can view F as a collection of \emph{fibre maps} $(f_\theta)_{\theta\in\T^1}$.
Further, for any $\omega\in\T^1$ we set $R_\omega(\theta,x)\=(\theta+\omega, x)$ and
\[\mathcal{F}_\omega:=\{f=R_\omega\circ F\ |\ F\in\mathcal F\}.\]

In the following, let $ f=R_\omega\circ F\in\mathcal F_\omega$ be given, where $\omega\in\T^1$ is irrational and $F\in\mathcal{F}$.
It is customary to use the notation
\[f_\theta^k(x):=\pi_2\circ f^k(\theta,x) \quad (\theta,x\in\T^1,\ k\in\Z).\]
In particular, this means $f_\theta^{-1}=(f_{\theta-\omega})^{-1}$.
We assume the existence of both an \emph{interval of contraction}
$C=[c^-, c^+]\ssq \T^1$ and \emph{expansion} $E=[e^-, e^+]\ssq \T^1$
where $C$ and $E$ are disjoint (the naming becomes clear below) and a finite union  $\mc I_{0}
\ssq \T^1$ of $\mc N$ disjoint open intervals $\mc I_0^1,\ldots, \mc I_0^{\mc N}$, called the \emph{(first) critical region}, such that
\begin{align}
  f_{\theta}\left(x\right)\in \textrm{int}(C) \text{ for all } x\notin (e^-, e^+) \text{
    and }\theta \notin \I_{0}.
\label{axiom: 2}
\end{align}
Further, we suppose there are $\alpha>4$ and $S>0$ such that
for arbitrary $\theta,\theta' \in \T^1$ we have
\begin{align}
  \alpha^{-2}d(x,x')\leq d(f_{\theta}(x),f_{\theta}(x'))&\leq \alpha^{2}
  d(x,x')\ \text{ for all } x,x'\in \T^1, \label{eq: lipschitz x}\\
  d(f_{\theta}(x),f_{\theta'}(x))&\leq S d(\theta,\theta')
 \text{ for all } x\in \T^1, \label{eq: lipschitz theta}\\
  |\partial_xf_{\theta}(x)|&\leq \alpha^{-1}\
  \label{eq: lipschitz x in C} \text{ for all } x\in C, \\
  |\partial_xf_{\theta}(x)|&\geq \alpha\
  \label{eq: lipschitz x in E} \text{ for all } x\in E.
\end{align}
These are the explicit estimates needed to define
$\mathcal{V}_\omega$. In order to state the required
dynamical properties, let $K_n=K_0 \kappa^n$ for some integers $\kappa\geq2, \ K_0\in \N$.
Set
\begin{align*}
 b_0\=1, \qquad b_n\=(1-1/K_{n-1})b_{n-1} \qquad (n\in \N)
\end{align*}
and assume $K_0$ and $\kappa$ are big enough such that
$b\=\lim_{n\to \infty} b_n>\sqrt{5/6}$.
\begin{defn}\label{defn: critical sets}
  Let $(M_n)_{n\in\N_0}$ be a super-exponentially increasing sequence of integers with $M_0\geq 2$.
  For $n\in \N_0$, we recursively define the $n+1$-th \emph{critical
    region} $\mc I_{n+1}$ in the following way:
\begin{itemize}
 \item $\mc A_{n} \=\left (\mc I_{n} - (M_n-1)\w\right)\times C$,
 \item $\mc B_{n} \=\left(\mc I_{n} + (M_n+1)\w\right)\times E$,
 \item $\mc I_{n+1}\=\textrm{int}\left( \pi_{1} \left(f^{M_n-1} (\A_{n})\cap f^{-(M_n+1)}(\B_{n})\right)\right)$.
\end{itemize}

Note that we trivially have $\I_{n+1}\ssq\I_n$.
For $n\in \N_0$, set  $\mc{W}_n^+ \= \bigcup_{j=0}^{n}
\bigcup_{l=1}^{M_j+1}\mc I_j+l\w$; $\mc{W}_n^- \= \bigcup_{j=0}^{n}
\bigcup_{l=-(M_j-1)}^{0}\mc I_j+l\w$ and set $\mc{W}_{-1}^\pm=\emptyset$.
\end{defn}

\begin{defn}\label{def: recurrence properties of critical intervals}
 Suppose $(M_n)_{n\in\N_0}$ and $(\mc I_n)_{n\in\N_0}$ are chosen as above with $M_{n+1}\leq 2 \alpha^{M_n/16}$ ($n\in\N_0$). Let $(\varepsilon_n)_{n\in\N_0}$ be
  a non-increasing sequence of positive real numbers satisfying $\varepsilon_0\leq 1$ and $\varepsilon_{n+1}
  \leq 2\alpha^{-M_n/4}/s$ for some fixed $s>0$ and all $n\in\N_0$.
  We say $f$
  verifies $(\mc F1)_n$ and $(\mc F2)_n$, respectively if
\begin{enumerate}[$(\mc F1)_n$]
\item $\mc I_{j}\cap \bigcup_{k=1}^{2K_jM_j} \mc I_{j}
    +k\w=\emptyset$, for all $j=0,\ldots, n$ \label{axiom: diophantine 1}
 \item $\left(\mc I_{j}- (M_j-1)\w \cup \mc I_{j}+ (M_j+1)\w\right) \bigcap
\left ( \mc{W}_{j-1}^+ \cup \mc{W}_{j-1}^-\right)
=\emptyset$, for all $j=1,\ldots, n$. \label{axiom: diophantine 2}
\end{enumerate}
If $f$ satisfies both
$(\mc F1)_n$, $(\mc F2)_n$, and $\I_j\neq \emptyset$ for $j=0,\ldots,n$, we say $f$ satisfies
$(\mc F)_n$. 
Further, for $(\varepsilon_n)_{n\in\N_0}$ as above,
 we say $f$ satisfies $(\mathcal{E})_n$ if $\mc I_n$ also consists of exactly $\mc N$ connected components $\mc I_n^1, \ldots, \mc I_n^{\mc N}$ with 
\begin{itemize}
\item\quad  $|\I_n^\iota|\ < \eps_n$ \label{axiom: I_n}  for all $\iota\in\{1,\ldots, \mc N\}$.
\end{itemize}
\end{defn}
\begin{rem}
If ${(M_n)}_{n\in\N}$, ${(\I_n)}_{n\in\N}$ and ${(\varepsilon_n)}_{n\in\N}$ full fill the assumptions of the above definition, then there exist $\alpha_*>1$ and $\eps_*>0$ such that
for any $\alpha\geq\alpha_*$ and $0<\eps_0\leq \eps_*$ we have
\begin{equation}\label{estimate-eps-m}
\textrm{Leb}_{\T^1}\left(\bigcup_{n\in\N}\bigcup_{\ell=0}^{M_n}\I_n-\ell\w\right)\leq \sum_{n=0}^\infty(M_n+1)\mc N\eps_n\leq \sum_{n=0}^\infty \eps_n^{1/2}<1/16.
\end{equation}
\end{rem}

In the following, we say $f$ \emph{satisfies
(\ref{axiom: 2})--(\ref{eq: lipschitz x in E}), $(\mathcal{F})_n$ and $(\mathcal{E})_n$} if it verifies the respective assumptions for some choice of the above constants and sequences
$(M_n)_{n\in\N_0}, (\varepsilon_n)_{n\in\N_0}$ with $\alpha\geq \alpha_*, 0<\eps_0\leq \eps_*$. 
With these notions, we are now in the position to define the set
$\mathcal{V}_\omega$ for any $\omega\in\T^1\setminus\Q$ and formulate our main result
Theorem~\ref{prop: dimensions_subgraphs}.
\begin{defn}
For any $\omega\in\T^1\setminus\Q$, we say $f\in\mathcal{F}_\omega$ is an element of $\mathcal{V}_\omega$ if
\begin{itemize}
\item $f$ satisfies (\ref{axiom: 2})--(\ref{eq: lipschitz x in E}) with $\alpha\geq \alpha_*$ and $|\I_0^\iota|<\eps_0\leq \eps_* (\iota=1,\ldots,\mc N)$;
\item $f$ satisfies $(\mathcal{F})_n$ and $(\mathcal{E})_n$ for all $n\in\N$;
\item $f$ has an SNA $\phi^+$ and SNR $\phi^-$, and $\mu_{\phi^\pm}$ are the only $f$-invariant ergodic measures;
\item $f$ is minimal.
\end{itemize}
\begin{thm}\label{prop: dimensions_subgraphs}
Suppose $f\in\mathcal{V}_\omega$. Then $\mu_{\phi^+}$ and $\mu_{\phi^-}$ are $1$-rectifiable.
\end{thm}
\end{defn}
We finish this section with two statements that highlight from different perspectives that--despite the technical character of the above assumptions--elements of $\mc V_\w$ occur naturally.
\begin{prop}[{cf. \cite[Theorem~2.1]{Jager}}]\label{prop: tobias cmp}
Given $\delta>0$, there exists a non-empty $\mathcal C^1$-open set $\mc U=\mathcal U(\delta)\subseteq \mathcal F$ with the following property.
For all $F\in\mathcal U$ there exists a set $\Delta_F\subseteq \T^1$ with
$\textrm{Leb}_{\T^1}(\Delta_F)\geq1-\delta$ and such that for any $\omega\in\Delta_F$ we have $R_\omega\circ F\in \mathcal{V}_\omega$.
\end{prop}

We may as well take another point of view and fix the rotation $R_\w$ while--as explained in the introduction--looking at whole families of maps in $\mathcal F_\w$.
Here as well, it turns out that members of these families typically lie in $\mc V_\w$.

More precisely, consider the following set of differentiable one-parameter families
\[\mathcal P:=\{(F_\tau)_{\tau\in\T^1}\ |\ F_\tau\in\mathcal F\ \textrm{and}\ (\tau,\theta,x)\mapsto F_\tau(\theta,x)\  \textrm{is}\ \mathcal C^1 \ \textrm{for all} \ \tau\in\T^1\}.\]
We say $\omega\in\T^1$ satisfies the \emph{Diophantine} condition with positive constants $\gamma,\nu$ if
\begin{equation}\label{condition_diophantine}
d(n\omega, 0)>\gamma\cdot |n|^{-\nu},\ \forall n\in\Z\setminus \{0\}.
\end{equation}
By $\mathcal D(\gamma,\nu)$, we denote the set of frequencies $\omega\in\T^1$ which satisfy  (\ref{condition_diophantine}). Then the following holds.

\begin{prop}[{\cite[Theorem~3.1]{jager2013}}]\label{prop: tobias etds}
 Given $\delta>0$ as well as $\gamma,\nu>0$, there exists a non-empty $\mathcal C^1$-open set $\mathcal U=\mathcal U(\gamma,\nu,\delta)\subseteq \mathcal P$
 with the following property.
 For all $(F_\tau)_{\tau\in\T^1}\in\mathcal U$ and all $\omega\in \mathcal D(\gamma,\nu)$
 there exists a set $\Lambda^{(F_\tau)}(\omega)\subseteq \T^1$ with $\textrm{Leb}_{\T^1}(\Lambda^{(F_\tau)}(\omega))\geq1-\delta$ and such that for any $\tau\in\Lambda^{(F_\tau)}(\omega)$
 we have $R_\omega\circ F_\tau\in \mathcal{V}_\omega$.
\end{prop}
\section{Rectifiability}
\label{sec: rectifiability}
From now on, we only consider the SNA $\phi^+$ of the map $f\in\mathcal{V}_\omega$ for $\omega\in\T^1\setminus\Q$.
All of the results and proofs which are only stated in terms of $\phi^+$ hold analogously for the repeller $\phi^-$
as can be readily seen by considering $f^{-1}$ instead of $f$.

Our analysis of the geometry of the measure supported on the SNA
relies on the fact that outside a Lebesgue null set, we can decompose $\phi^+$ in countably many
Lipschitz graphs.
Let us briefly sketch the argument for the existence of such a decomposition.

For given $\theta_0,\theta_1\in \T^1$, observe that the invariance of
$\phi^+$ trivially implies
\begin{align}\label{eq: intuitive idea}
d(\phi^+(\theta_0),\phi^+(\theta_1))=d(f_{\theta_0-n\w}^n\left(\phi^+(\theta_0-n\w)\right),f_{\theta_1-n\w}^n\left(\phi^+(\theta_1-n\w)\right)).
\end{align}
For simplicity, let us discuss the case $n=1$. Clearly,
\begin{align*}
d(f_{\theta_0-\w}\left(\phi^+(\theta_0-\w)\right),f_{\theta_1-\w}\left(\phi^+(\theta_1-\w)\right))&\leq
d(f_{\theta_0-\w}\left(\phi^+(\theta_0-\w)\right),f_{\theta_0-\w}\left(\phi^+(\theta_1-\w)\right))\\
&\phantom{\leq }+
d(f_{\theta_0-\w}\left(\phi^+(\theta_1-\w)\right),f_{\theta_1-\w}\left(\phi^+(\theta_1-\w)\right)).
\end{align*}
Equation \eqref{eq: lipschitz theta} yields that the second summand is bounded by $Sd(\theta_0,\theta_1)$
while \eqref{eq: lipschitz x in C} gives that the first one can be considered small (less than $\alpha^{-1}$)
whenever $\phi^+(\theta_i-\w)\in C$ ($i=0,1$).
In view of \eqref{eq: intuitive idea}, this suggests that in order to get Lipschitz continuity of $\phi^+$ over some subset $\Omega \ssq \T^1$,
we have to ensure that big portions of the orbit segments $\{\phi^+(\theta_i-n\w),\ldots,\phi^+(\theta_i-\w)\}$ ($i=0,1$)
lie in $C$ for each two $\theta_0,\theta_1\in \Omega$.
As $\lam(\phi^+)<0$, \emph{most} parts of $\phi^+$ have to lie in $C$ so that
for almost all $\theta_0,\theta_1\in \T^1$ there should be a strictly increasing sequence $n_\ell$ with
$\phi^+(\theta_0-n_\ell),\phi^+(\theta_1-n_\ell)\in C$.
Now, according to \eqref{axiom: 2}, a natural obstruction for the segments $\{\phi^+(\theta_i-n_\ell\w),\ldots,\phi^+(\theta_i-\w)\}$ (which start in $C$)
to have a large intersection with $C$
is a high frequency of visits to the critical region.

However, when restricting to sets
\begin{align*}
\Omega_j= \T^1 \setminus \bigcup_{k=j}^{\infty} \bigcup_{l=0}^{2K_kM_k} \I_k+l\w \qquad(j\in \N),
\end{align*}
we can derive sufficient upper bounds for these frequencies.

Observe that $K_kM_k\leq2 K_0 \kappa^k \cdot \alpha^{M_{k-1}/16}$
while $|\I_k^\iota|<\eps_k\leq 2 \alpha^{-M_{k-1}/4}/s$.
Having in mind that $\I_k$ consists of $\mc N$ connected components $\I_k^\iota$ and that $M_k$ grows super-exponentially,
we easily get the following rough estimate
\begin{align}\label{eq: measure of Omegaj}
\operatorname{Leb}_{\T^1}\left(\bigcup_{k=j}^\infty\bigcup_{l=0}^{2K_kM_k}\I_k+l\w\right)<
 \sum_{k=j}^\infty (2K_kM_k+1) \mc N\eps_k <\sum_{k=j}^\infty \eps_k^{1/2},
\end{align}
and hence $\operatorname{Leb}_{\T^1} (\Omega_j)>0$ for large enough $j$.

We still have to take care of the complement of the $\Omega_j$
\begin{align*}
 \Omega_\infty = \T^1\setminus \bigcup_{j\in \N} \Omega_j=\bigcap_{i=1}^{\infty}
\bigcup_{k=i}^{\infty}
\bigcup_{l=0}^{2K_kM_k}\I_k+l\w.
\end{align*}
However, due to \eqref{eq: measure of Omegaj}, we have $\operatorname{Leb}_{\T^1}\left(\Omega_\infty\right)=0$.

The next proposition is the basis of all our investigation of $\phi^+$ in this work.
Its proof is given in the last section.
However, the statement should seem plausible to the reader in the light of the above discussion.
\begin{prop}\label{prop: decomposition in lipschitz graphs}
Let $f\in \mc V_\omega$. There is a $\textrm{Leb}_{\T^1}$-null set $\mc M$ and there are $L_j>0$ $(j\in\N)$
such that the following is true.
If $\theta,\theta'\in \Omega_j\setminus \mc M$, then $\left|\phi^+(\theta)-\phi^+(\theta')\right|\leq L_j
  d(\theta,\theta')$.
\end{prop}

\label{sec: hausdorff and pointwise dimension}
Now, taking this statement for granted, we straightforwardly get our main result (cf. \cite{GroegerJaeger2012SNADimensions}).
\begin{proof}[Proof of Theorem~\ref{prop: dimensions_subgraphs}]
  For each $j\in\N\cup\{\infty\}$ set
  $\psi_j\=\left.\phi^+\right|_{\tilde\Omega_j}$, where $\tilde \Omega_j=\Omega_j\setminus \mc M$ ($j\in \N$) and $\tilde \Omega_\infty=\Omega_\infty\cup \mc M$.  First, we want to show
  that the graph $\Psi_j=\{(\theta,\psi_j(\theta))\:\theta\in\tilde\Omega_j\}$ is the image of a bi-Lipschitz continuous
  function $g_j$ for all $j\in\N$.

  Define
  $g_j:\tilde\Omega_j\to\tilde\Omega_j\times\T^1$ via
  $\theta\mapsto(\theta,\psi_j(\theta))$ for all
  $j\in\N$. We have that $g_j(\tilde \Omega_j)=\Psi_j$ and
  $d(g_j(\theta),g_j(\theta'))\geq d(\theta,\theta')$
  for all $\theta,\theta'\in\tilde\Omega_j$.
 Further, Proposition~\ref{prop: decomposition in lipschitz graphs} yields
  $d(\psi_j(\theta),\psi_j(\theta'))=d(\phi^+(\theta),\phi^+(\theta'))<L_j d(\theta,\theta')$
  for all $\theta,\theta'\in\tilde\Omega_j$.
 Hence, $g_j$ is bi-Lipschitz continuous for each $j\in\N$.

Now, by definition, $\mu_{\phi^+}$ is absolutely continuous
with respect to $\left.\mathcal H^1\right|_{\Phi^+}$. We have that
$\mu_{\phi^+}(\Psi_\infty)=0$ and therefore $\mu_{\phi^+}$ is also
absolutely continuous with respect to $\left.\mathcal
  H^1\right|_{\Phi^+\backslash\Psi^\infty}$. Since
$\Phi^+\backslash\Psi_\infty=\bigcup_{j\in \N}\Psi_j$ is a
countably $1$-rectifiable set we get that $\mu_{\phi^+}$ is $1$-rectifiable,
too.
\end{proof}

\section{Proof of Proposition~\ref{prop: decomposition in lipschitz graphs}}
We now turn to the proof of Proposition~\ref{prop: decomposition in lipschitz graphs}.
It is based on both the $\mc C^1$-estimates and the dynamical assumptions that define the set $\mc{V}_\omega$ (see Section~\ref{subsec: basic setting and notation}).
Recall that we consider the fixed map $f=R_\omega\circ F\in\mathcal V_\omega$, where $\omega\in\T^1\setminus\Q$ and $F\in\mathcal F$.
As before, we only consider the SNA $\phi^+$ of $f$.

A crucial point in our analysis is to control the frequency of visits a forward orbit pays to the interval of contraction.
We hence study the following quantities for $n,N\in \N$
\begin{align*}
\mc P_n^N(\theta,x) &= \# \{\ell \in [n,N-1]\cap \N_0\:f_{\theta}^\ell(x) \in C \text{ and } \theta+\ell\w\notin \I_0\}.
\end{align*}
In order to get lower bounds on the $\mc P_n^N(\theta,x)$ for certain $\theta$ and $x$, we have to apply a number
of combinatorial lemmas.
Their proofs can be found in \cite{fuhrmann2014,fuhrmanngrogerjager14}.

In the following, let $\mc Z^-_n\= \bigcup_{j=0}^{n}
\bigcup_{l=-(M_j-2)}^{0}\mc I_j+l\w$ for $n\in\N_0$ and set, for the sake of a convenient notation, $M_{-1}\=0$,  $\I_{-1}\=\I_0$, as well as $\mc Z^-_{-1}\=\emptyset$.
\begin{defn}
We say that $(\theta,x)$ verifies $(\mc B1)_n$
if
\begin{enumerate}[$(\mc B1)_n$]
 \item $x \in C$ and $\theta \notin \mc Z^-_{n-1}$. \label{axiom: B1}
\end{enumerate}
\end{defn}
\begin{lem}[{cf. \cite[Lemma~4.4]{fuhrmann2014}}]
\label{lem: duration of stay in contracting/expanding regions}
Let $f\in \mathcal{V}_\omega$, $n\in \N_0$ and assume $(\theta,x)$ satisfies
$(\mc B1)_n$.
Let $\mc L$ be the first time $l$
such that $\theta+l\w \in \I_n$
and let $0<\mc L_{1}<\ldots <\mc L_{N}=\mc L$ be all those
times $m\leq \mc L$ for which $\theta+m \w \in \I_{n-1}$. Then
$f^{\mc L_{i}+M_{n-1}+2}(\theta,x)$ satisfies
$(\mc B1)_n$ for each $i=1,\ldots,N-1$ and the following implication holds
\begin{align*}
 f^k_\theta(x) \notin C \Rightarrow \theta+k\w \in \mc{W}_{n-1}^+
\quad (k = 1,\ldots, \mc L).
\end{align*}
\end{lem}

\begin{lem}[{cf. \cite[Lemma~4.8]{fuhrmann2014}}]\label{lem: estimate for times spent in contracting/expanding regions}
Let $f\in \mc{V}_\omega$ and assume $(\theta,x)$ verifies $(\mc B1)_{n}$ for $n\in \N$. Let $0<\mc L_1<\ldots <\mc L_N=\mc L$ be as in Lemma~\ref{lem: duration of stay in contracting/expanding regions}. Then, for each $i=1,\ldots,N$, we have
\begin{align}\label{eq: Lemma 4.8}
 \mathcal P_k^{\mc L_i}(\theta,x) \geq b_n (\mc L_i-k) \quad (k=0,\ldots,\mc L_i-1).
\end{align}
\end{lem}

Let $p\in \N$ and consider a finite orbit $\{(\theta_0,x),\ldots, f^n(\theta_0,x)\}$ which initially verifies $(\mc B1)_p$ and hits $\mc I_p$ only at $\theta_0+n\w$. Lemma~\ref{lem: estimate for times spent in contracting/expanding regions} provides us with a lower bound on the times spent in the contracting region between
any time $k$ and only such following times at which the orbit hits $\I_{p-1}$.
If we want a lower bound on the times in the contracting region between any two consecutive moments $k<l$, we have to deal with the fact that Lemma~\ref{lem: duration of stay in contracting/expanding regions} might allow the orbit to stay in the expanding region for $M_{p-1}+1$ times after hitting $\I_{p-1}$. This is taken care of in the following corollary of Lemma~\ref{lem: duration of stay in contracting/expanding regions} and Lemma~\ref{lem: estimate for times spent in contracting/expanding regions}.

For $\theta\in \T^1$ and $0\leq k \leq n$, set
\begin{align*}
p_k^n(\theta)=
\max\left\{p\in \N_{0}\:
\exists l \in \left[M_{p-1},\min\left\{n,n-k+M_{p}+1\right\}\right] \text{ such that }
\theta-l\w \in \I_p\right\}
\end{align*}
with $\max \emptyset \=-1$.

\begin{cor}[{cf. \cite[Corollary~5.4]{fuhrmanngrogerjager14}}]\label{cor: times in C}
Let $f\in \mc{V}_\omega$ and suppose
$(\theta-n\w,x)$ satisfies $(\B1)_{p_0^n(\theta)+1}$.
Then
\begin{align}\label{eq: times in contracting region arbitrary successive times}
\mc P_{k}^{n}(\theta-n\w,x)\geq b_{p_k^n(\theta)+1} \left(n-k-\sum_{j=0}^{p_k^n(\theta)} (M_{j}+2)\right)\quad \text{for each }k=0,\ldots,n-1.
\end{align}
\end{cor}
We need one more combinatorial ingredient, in order to control $p_k^n(\theta)$.
Let us introduce $i_k^n:=\max\{l\:n-k\geq 2K_lM_l-M_l-1\}$ for $k, n \in \N$.
\begin{prop}[{cf. \cite[Proposition~5.5]{fuhrmanngrogerjager14}}]\label{prop: pk < ik}
Suppose $\theta \in \Omega_j$ for some $j\in \N$. Then $i_k^n\geq p_k^n(\theta)$ for all
$0\leq k\leq n- (2K_{j-1}M_{j-1}-M_{j-1}-1)$.
\end{prop}
\begin{proof}
Note that by the assumptions $i_k^n\geq j-1$.
Thus, without loss of generality we may assume $p_k^n(\theta)> j-1$.
By definition of $p_k^n(\theta)$, there is $l\in \left[M_{p_k^n(\theta)-1},n-k+M_{p_k^n(\theta)}+1\right]$ such that $\theta-l\w \in \mc I_{p_k^n(\theta)}$.
Since $\theta \in \Omega_j$, this implies
$l>2K_{p_k^n(\theta)}M_{p_k^n(\theta)}$ and thus,
$n-k>2K_{p_k^n(\theta)}M_{p_k^n(\theta)}-M_{p_k^n(\theta)}-1$ which means
$i_k^n\geq p_k^n(\theta)$.
\end{proof}
As the SNA $\phi^+$ is attracting, we expect it to share a big intersection with the interval of contraction.
The next statement confirms this expectation.
\begin{prop}\label{prop: b}
 Consider a representative $\phi^+$ of the equivalence class of the SNA.
Then
\[
\textrm{Leb}_{\T^1}(\{\theta\:\phi^+(\theta)\notin E\})\geq
b-1/3.
\]
\end{prop}
\begin{proof}
Since all critical sets $\mathcal I_n$ are non-void, the same is true for the sets $\textrm{cl}\left(f^{M_n}(\mathcal A_n)\right)$ (cf. Definition~\ref{defn: critical sets}).
As a consequence of Lemma \ref{lem: duration of stay in contracting/expanding regions} 
and $(\mathcal F2)_n$, the latter form a nested sequence of compact sets such that their intersection is non-void as well.
Let $(\theta,x)\in\bigcap_{n\in\N}\textrm{cl}\left(f^{M_n}(\mathcal A_n)\right)$. Then the point $(\theta',x')\=f^{-M_n}(\theta,x)$ satisfies $(\mathcal B1)_n$ and 
$f_{\theta'}^{M_n-1}(x')\in C$ by Lemma \ref{lem: duration of stay in contracting/expanding regions}. Hence, for any $k\in[0,M_n]$ we have 
\begin{eqnarray*}
\partial_x f_{\theta}^{-k}(x)&=&\frac{1}{\partial_x f_{\theta-k\omega}^{k}(f_\theta^{-k}(x))}=\frac{1}{\prod_{j=\mathcal L-k+1}^{\mathcal L}\partial_x f_{\theta'+j\omega}(f_{\theta'}^j(x'))}\\
&\stackrel{(\ref{eq: lipschitz x}), (\ref{eq: lipschitz x in C})}{\geq}& \alpha\cdot \alpha^{\mathcal P_{\mathcal L-k+1}^{\mathcal L}(\theta',x')}
\alpha^{-2(k-1-\mathcal P_{\mathcal L-k+1}^{\mathcal L}(\theta',x'))}\stackrel{\text{Lemma~\ref{lem: estimate for times spent in contracting/expanding regions}}}{\geq}
\alpha_-^{-k},
\end{eqnarray*}
where $\mathcal L=M_n-1$ and $\alpha_-=\alpha^{-(3b-2)}<1$. As $M_n\nearrow \infty$, 
 the point $(\theta,x)$ verifies
\[
 \lambda^-(\theta,x):=\limsup_{k\to\infty} 1/k\cdot \log \d_x f^{-k}_\theta(x)\geq-\log \alpha_-,
\]
where $\lambda^-(\theta,x)$ is the \emph{backwards Lyapunov} exponent of the point $(\theta,x)$.
Now, by the Semi-uniform Birkhoff Ergodic Theorem (see \cite[Theorem~1.9]{sturmanstark2000}) we know that if the Lyapunov exponents for all 
invariant measures (which, in the present situation, are given by the Lyapunov exponents of the invariant graphs)
are smaller than a constant $a$, then all pointwise Lyapunov exponents are uniformly bounded below $a$.
By the definition of $\mc V_\w$, $\phi^+$ gives rise to the only invariant ergodic measure with a negative Lyapunov exponent so that this 
observation--applied to the inverse map $f^{-1}$--yields $\lam(\phi^+)\leq\log \alpha_-$. 
Due to \eqref{eq: lipschitz x} and \eqref{eq: lipschitz x in E}, this gives
\[
\textrm{Leb}_{\T^1}(\{\theta\:\phi^+(\theta)\notin E\}) \log \alpha^{-2}+(1-\textrm{Leb}_{\T^1}(\{\theta\:\phi^+(\theta)\notin E\})) \log \alpha\leq \log \alpha_-,
\]
proving the statement.
\end{proof}
In the following, let $\mc M\ssq\T^1$ comprise those $\theta$ whose backwards orbits (under $R_\w$) visit at least one of the sets
$\bigcup_{n\in\N} \bigcup_{j=0}^{M_n} \mc I_n-j\w$ and ${{\phi^+}^{\vphantom{-}}}^{-1}(E)$ with a frequency different from the respective
Lebesgue measure.
Observe that Birkhoff's Ergodic Theorem implies that $\mc M$ is a $\textrm{Leb}_{\T^1}$-null set.

\begin{proof}[Proof of Proposition~\ref{prop: decomposition in lipschitz graphs}]
For this proof, we refer by $|I|$ to the length (and in contrast to the previous convention not to the diameter)
of subsets $I\ssq\T^1$.
Let $\theta,\theta' \in \Omega_j\setminus \mc M$ and assume without loss of generality that $d(\theta,\theta')<|E|/(4S)$.
Note that there is a strictly increasing sequence $(\tilde n_\ell)$ such that
$\theta-\tilde n_\ell\w,\theta'-\tilde n_\ell\w \notin \bigcup_{n\in\N} \bigcup_{m=0}^{M_n} \I_n-m\w$ as well as $\phi^+(\theta-\tilde n_\ell\w),\phi^+(\theta'-\tilde n_\ell\w) \notin E$ because
\begin{align*}
&\lim_{m\to\infty} \frac{1}{m} \sum_{\ell=0}^{m-1}\left( \mathbf{1}_{\bigcup_{n\in\N} \bigcup_{j=0}^{M_n} \mc I_n-j\w}(\theta-\ell\w)+
  \mathbf{1}_{\bigcup_{n\in\N} \bigcup_{j=0}^{M_n} \mc I_n-j\w}(\theta'-\ell\w)\right.\\
  &  \phantom{\lim_{m\to\infty} \frac{1}{m} \sum_{\ell=0}^{m-1}}+  \mathbf{1}_{{{\phi^+}^{\vphantom{i}}}^{-1}(E)}(\theta-\ell\w)+ \mathbf{1}_{{{\phi^+}^{\vphantom{i}}}^{-1}(E)}(\theta'-\ell\w) \bigg)\\
 &=2 \cdot \textrm{Leb}_{\T^1}\left(\bigcup_{n\in\N} \bigcup_{\ell=0}^{M_n} \mc I_n-\ell\w\right)+ 2\cdot \textrm{Leb}_{\T^1}\left({{\phi^+}^{\vphantom{i}}}^{-1}(E)\right)<1,
\end{align*}
where we used \ref{estimate-eps-m} and Proposition~\ref{prop: b} in the last step.
Given such $\tilde n_\ell$, observe that
$\theta-(\tilde n_\ell-1)\w,\theta'-(\tilde n_\ell-1)\w \notin \bigcup_{n\in\N} \mc Z_n^{-}$ as well as $\phi^+(\theta-(\tilde n_\ell-1)\w),\phi^+(\theta'-(\tilde n_\ell-1)\w) \in C$, 
due to \eqref{axiom: 2}.
We set $n_\ell\=\tilde n_\ell-1$ and hence have that $(\theta-n_\ell\w,\phi^+(\theta-n_\ell\w))$ and $(\theta'-n_\ell\w,\phi^+(\theta'-n_\ell \omega))$ satisfy $(\mc B1)_{p_0^{n_\ell}(\theta)+1}$ and $(\mc B1)_{p_0^{n_\ell}(\theta')+1}$ respectively.

By Corollary~\ref{cor: times in C} and Proposition~\ref{prop: pk < ik}, we thus get
\begin{align}\label{eq: 20}
\begin{split}
\mc P_{k}^{n_\ell}(\theta-n_\ell\w,\phi^+(\theta-n_\ell\w))\geq
&b_{p_k^{n_\ell}(\theta)+1} \left(n_\ell-k-\sum_{m=0}^{p_k^{n_\ell}(\theta)} (M_{m}+2)\right)\\
\stackrel{\text{Proposition~\ref{prop: pk < ik}}}{\geq}
&b_{i_k^{n_\ell}+1} \left(n_\ell-k-\sum_{m=0}^{i_k^{n_\ell}} (M_{m}+2)\right),
\end{split}
\end{align}
for $0\leq k\leq n_\ell - (2K_{j-1}M_{j-1}-M_{j-1}-1)$.

Without loss of generality, we may assume that $j$ is large enough so that $\sum_{m=0}^{i_k^{n_\ell}} (M_{m}+2)\leq \frac{3}{2} M_{i_k^{n_\ell}}$
(note that $i_k^{n_\ell}\geq j-1$). Further, $({n_\ell}-k)/K_{i_k^{n_\ell}}\geq
2M_{i_k^{n_\ell}}-M_{i_k^{n_\ell}}/K_{i_k^{n_\ell}}-1/K_{i_k^{n_\ell}}$ by definition of
$i_k^{n_\ell}$.
Thus, we
have $\sum_{m=0}^{i_k^{n_\ell}} (M_{m}+2)\leq ({n_\ell}-k)/K_{i_k^{n_\ell}}$ and so by \eqref{eq: 20}
\begin{align}\label{eq: P n-k n}
 \mc P_{k}^{{n_\ell}}(\theta-{n_\ell}\w,\phi^+(\theta-n_\ell\w))\geq b_{i^{n_\ell}_k+1} (1-1/K_{i_k^{n_\ell}}) ({n_\ell}-k) > b^2({n_\ell}-k).
\end{align}
A similar estimate holds true for $\theta'$.

Now, given $\theta,\theta'\in \T^1$ and $n\in\N_0$, set
\begin{align*}
\wp^n(\theta,\theta')=
\#\left\{-1\leq m <n-1\ \left|\ \phi^+(\theta+m\w),\phi^+(\theta'+m\w) \in C\text{ and } \theta+m\w, \theta'+m\w \notin \I_0 \right.\right\}
\end{align*}
and observe that if $\wp^1(\theta,\theta')=1$, then both $\phi^+(\theta)$ and $\phi^+(\theta')$ lie in $C$ due to \eqref{axiom: 2}.
By induction on $n$, we next show that for all $n\in\N$
\begin{align}\label{eq: induction step n}
\begin{split}
d\left(\phi^+(\theta+n\w),\phi^+(\theta'+n\w)\right)&\leq
\alpha^{2n-3\wp^n(\theta,\theta')}
d\left(\phi^+(\theta),\phi^+(\theta')\right)
\\
&\phantom{\leq}+
S d(\theta,\theta') \sum_{k=1}^{n} \alpha^{2(n-k)-3\wp^{n-k}(\theta+k\w,\theta'+k\w)}.
\end{split}
\end{align}
First, we get
\begin{align}\label{eq: induction step 1}
\begin{split}
d\left(\phi^+(\theta+\w),\phi^+(\theta'+\w)\right)&\leq
d\left(f_{\theta}(\phi^+(\theta)),f_{\theta}(\phi^+(\theta'))\right)+
d\left(f_{\theta}(\phi^+(\theta')),f_{\theta'}(\phi^+(\theta'))\right)\\
&\leq
 \alpha^{2 \left(1-\wp^1(\theta,\theta')\right)-\wp^1(\theta,\theta')} d\left(\phi^+(\theta),\phi^+(\theta')\right)
+
S d(\theta,\theta').
\end{split}
\end{align}
To see this, we may assume without loss of generality that $\wp^1(\theta,\theta')=1$.
Then,
$\phi^+(\theta-\w)$ and $\phi^+(\theta'-\w)$ as well as $\phi^+(\theta)$ and $\phi^+(\theta')$ lie in $C$. 
Denote by $I'$ the line segment entirely contained in $C$ which connects $\phi^+(\theta-\w)$ and $\phi^+(\theta'-\w)$.\footnote{Note that
the length of $I'$ may not coincide with the distance of $\phi^+(\theta-\w)$ and $\phi^+(\theta'-\w)$
in $\T^1$.}
We have that $f_{\theta-\w}(I')\ssq C$ [due to \eqref{axiom: 2}] and $|f_{\theta-\omega}(I')|\leq \alpha^{-1} |C| <|C|/4$
[due to \eqref{eq: lipschitz x in C}].
If we denote by $I\ssq C$ that line segment which connects $\phi^+(\theta)$ and $\phi^+(\theta')$, observe that
$I$ is contained in an $|E|/4$-neighbourhood of $f_{\theta-\omega}(I')$
since 
\[
 d\left(f_{\theta-\w}(\phi^+(\theta'-\w)),\phi^+(\theta')\right)=d\left(f_{\theta-\w}(\phi^+(\theta'-\w)),f_{\theta'-\w}(\phi^+(\theta'-\w))\right)\leq S d(\theta,\theta')<|E|/4.
\]
In particular, this implies $|I|<1/2$ so that $d(f_\theta(\phi^+(\theta)), f_{\theta}(\phi^+(\theta')))\leq \alpha^{-1}d(\phi^+(\theta),\phi^+(\theta'))$
due to \eqref{eq: lipschitz x in C} which proves \eqref{eq: induction step 1}.

Note that (\ref{eq: induction step 1}) coincides with \eqref{eq: induction step n} for $n=1$.
Now, suppose \eqref{eq: induction step n} holds for some $n\in \N$.
Since $\wp^n(\theta,\theta')+\wp^{1}(\theta+n\w,\theta'+n\w)=\wp^{n+1}(\theta,\theta')$, we have
\begin{align*}
& d\left(\phi^+(\theta+(n+1)\w),\phi^+(\theta'+(n+1)\w)\right)\\
&=
d\left(f_{\theta+n\w}\left(\phi^+(\theta+n\w)\right),f_{\theta'+n\w}\left(\phi^+(\theta'+n\w)\right)\right)\\
&\leq
\alpha^{2\left(1-\wp^{1}(\theta+n\w,\theta'+n\w)\right)-\wp^{1}(\theta+n\w,\theta'+n\w)}
d\left(\phi^+(\theta+n\w),\phi^+(\theta'+n\w)\right)+Sd\left(\theta,\theta'\right)\\
&\leq
\alpha^{2(n+1)-3\wp^{n+1}(\theta,\theta')}
d\left(\phi^+(\theta),\phi^+(\theta')\right)
+ S d(\theta,\theta') \sum_{k=1}^{n+1}
\alpha^{2(n+1-k)-3\wp^{n+1-k}(\theta+k\w,\theta'+k\w)}
\end{align*}
where we used a similar argument as for \eqref{eq: induction step 1} and the induction hypothesis.
Hence, equation \eqref{eq: induction step n} holds.

Now, consider sufficiently large $j$ and $\theta,\theta'\in \Omega_j\setminus \mc M$ as above.
Suppose $n_\ell >2K_{j-1}M_{j-1}-M_{j-1}-1$ and observe that
equation \eqref{eq: P n-k n} gives
\begin{align*}
&\wp^{n_\ell-k}(\theta-(n_\ell-k)\w,\theta'-(n_\ell-k)\w)\\
&\geq n_\ell-k-\left(2(n_\ell-k)-\mc P_{k}^{n_\ell}(\theta-n_\ell\w)-\mc P_{k}^{n_\ell}(\theta'-n_\ell\w)\right)-2\\
& \geq
n_\ell-k-2(1-b^2)(n_\ell-k)-2=(2b^2-1)(n_\ell-k)-2
\end{align*}
for all $k=0,\ldots,n_\ell - (2K_{j-1}M_{j-1}-M_{j-1}-1)$.
Plugging this into (\ref{eq: induction step n}) and sending $\ell\to\infty$ yields
$|\phi^+(\theta)-\phi^+(\theta')|\leq L_j d(\theta,\theta')$
where
\begin{align}\label{eq: L-j upper bound}
L_j=S \sum_{k=2K_{j-1}M_{j-1}-M_{j-1}-1}^\infty \alpha^{6-c_0k}+S\sum_{k=0}^{2K_{j-1}M_{j-1}-M_{j-1}-2}\alpha^{2k}<\infty,
\end{align}
with $c_0=6b^2-5>0$.
\qedhere
\end{proof}

\bibliography{Literaturnachweis_SNA}{}
\bibliographystyle{unsrt}
\end{document}

%% file: bibstandard.tex
\newcommand{\w}{\omega}

\renewcommand{\d}{\ensuremath{\partial}}

\newcommand{\mc}{\mathcal}

\newcommand{\alphlist}{\begin{list}{(\alph{enumi})}{\usecounter{enumi}}}
\newcommand{\romanlist}{\begin{list}{(\roman{enumi})}{\usecounter{enumi}}}
\newcommand{\listend}{\end{list}}

\renewcommand{\:}{\colon}
\renewcommand{\=}{\coloneqq}

\newcommand{\ssq}{\ensuremath{\subseteq}}

\newcommand{\eps}{\ensuremath{\varepsilon}}

\newcommand{\lam}{\ensuremath{\lambda}}

\newcommand{\T}{\ensuremath{\mathbb{T}}}

\newcommand{\N}{\ensuremath{\mathbb{N}}} 
\newcommand{\R}{\ensuremath{\mathbb{R}}}
\newcommand{\Z}{\ensuremath{\mathbb{Z}}}
\newcommand{\Q}{\ensuremath{\mathbb{Q}}}

\newcommand{\kreis}{\ensuremath{\mathbb{T}^{1}}}

\newcommand{\I}{\ensuremath{\mathcal I}}
\newcommand{\A}{\ensuremath{\mathcal A}}
\newcommand{\B}{\ensuremath{\mathcal B}}

